\documentclass[a4paper,12pt]{article}

\makeatletter
\@addtoreset{footnote}{page}
\makeatother

\usepackage{enumitem}

\usepackage{amsthm}
\usepackage{amsmath,amssymb,latexsym,amsfonts,mathrsfs}

\renewcommand{\tilde}{\widetilde}

\newcommand{\floor}[1]{\!\left\lfloor #1 \right\rfloor \!} 

\newcommand{\bC}{\ensuremath{\mathbb{C}}}

\newcommand{\bE}{\ensuremath{\mathbb{E}}}

\newcommand{\bN}{\ensuremath{\mathbb{N}}}

\newcommand{\bP}{\ensuremath{\mathbb{P}}}

\newcommand{\bR}{\ensuremath{\mathbb{R}}}

\newcommand{\cG}{\ensuremath{\mathcal{G}}}

\newcommand{\cL}{\ensuremath{\mathcal{L}}}

\newcommand{\cP}{\ensuremath{\mathcal{P}}}

\theoremstyle{plain}
\newtheorem{Thm}{Theorem}[section]

\newtheorem{Prop}[Thm]{Proposition}
\newtheorem{Cor}[Thm]{Corollary}

\theoremstyle{definition}

\newtheorem{Rem}[Thm]{Remark}

\numberwithin{equation}{section}

\makeatletter
\renewcommand\section{\@startsection {section}{1}{\z@}%
                                   {-3.5ex \@plus -1ex \@minus -.2ex}%
                                   {2.3ex \@plus.2ex}%
                                   {\normalfont\large\bf}}
\makeatother

\makeatletter
\renewcommand\subsection{\@startsection {subsection}{1}{\z@}%
                                   {-3.5ex \@plus -1ex \@minus -.2ex}%
                                   {2.3ex \@plus.2ex}%
                                   {\normalfont\normalsize\bf}}
\makeatother

\begin{document}
\begin{center}
	{\Large \bf 
		Renewal dynamical approach for non-minimal quasi-stationary distributions of one-dimensional diffusions
	}
\end{center}
\begin{center}
	Kosuke Yamato
\end{center}
\begin{center}
	{\small \today}
\end{center}

\begin{abstract}
	We consider quasi-stationary distributions for one-dimensional diffusions via the renewal dynamical approach. We show that convergence of the iterative renewal transform to quasi-stationary distributions is equivalent to a condition on the moment growth rate of the lifetime,
	which is at the same time a necessary condition for the existence of Yaglom limits.
\end{abstract}

\section{Introduction}

	Let $X$ be an irreducible one-dimensional diffusion on the half-line $(0,\infty)$ stopped at $0$ which hits $0$ with probability one: $\bP_{x}[T_0 < \infty] = 1 \ ( x > 0 )$. Here $T_0$ denotes the first hitting time of $0$ and $\bP_x$ denotes the underlying probability measure of $X$ starting from $x$.
	A probability measure $\nu$ $(0,\infty)$ is called a \textit{quasi-stationary distribution} of $X$ when the following holds:
	\begin{align}
		\bP_{\nu}[X_t \in dx \mid T_0 > t] = \nu(dx) \quad (t > 0), \label{} 
	\end{align}
	and $\nu$ is called a \textit{Yaglom limit} of a probability measure $\mu$ on $(0,\infty)$ when
	\begin{align}
		\mu_{t}(dx) := \bP_{\mu}[X_{t} \in dx \mid T_{0} > t] \xrightarrow[t \to \infty]{} \nu(dx). \label{qsdConvergence}
	\end{align}
	Here and hereafter all the convergence of probability distributions is in the sense of the weak convergence.
	For a probability measure $\mu$ on $(0, \infty)$, we introduce the \textit{renewal transform} of $\mu$ by
	\begin{align}
		\Phi \mu (dx) := \frac{1}{\bE_{\mu}T_0} \int_{0}^{\infty}\bP_{\mu}[X_t \in dx, T_{0} > t]dt, \label{renewalTransform} 
	\end{align}
	which is the $0$-potential measure of $X$ normalized to be a probability measure under the initial distribution $\mu$, and can be defined when $\bE_{\mu}T_0 < \infty$.
	The distribution $\Phi \mu$ can be interpreted as the limit distribution of the conservative stochastic process which behaves as $X$ until $T_0$ and as soon as 
	it hits $0$, it jumps into a random point in $(0,\infty)$ according to the probability $\mu$ and starts afresh (see e.g., Ben-Ari and Pinsky \cite{JumpBoundary}).
	The renewal transform has a close relation to quasi-stationary distributions since every quasi-stationary distribution $\nu$ of $X$ is a fixed point of the renewal transform. Indeed, by the Markov property, it holds
	\begin{align}
		\Phi \nu (dx) &:= \frac{1}{\bE_{\nu}T_0} \int_{0}^{\infty}\bP_{\nu}[X_t \in dx \mid T_0 > t] \bP_{\nu}[T_0 > t]dt \label{} \\
		&= \frac{\nu(dx)}{\bE_{\nu}T_0} \int_{0}^{\infty}\bP_{\nu}[T_0 > t]dt = \nu(dx). \label{eq6}
	\end{align}
	This fact leads us to study Yaglom limits through the iterative renewal transform.

	For some diffusions, it is known that there exist infinitely many quasi-stationary distributions,
	and in the case, the set of quasi-stationary distributions is totally ordered by a usual stochastic order (see Proposition \ref{prop:qsdOrdering} for the precise description).
	The minimum element is called the \textit{minimal quasi-stationary distribution}.

	Our main objective in the present paper is to investigate the convergence of the iterative renewal transform to non-minimal quasi-stationary distributions.
	The reason is that as we will see in Theorem \ref{thm:hierarchyOfConds},
	for a probability measures $\mu$ and $\nu$
	the convergence of the iterative renewal transform 
	\begin{align}
		\Phi^{n}\mu \xrightarrow[n \to \infty]{} \nu \label{qsdIRT}
	\end{align}
	is a necessary condition for the convergence \eqref{qsdConvergence},
	and there have been many detailed studies on the convergence \eqref{qsdConvergence} to the minimal quasi-stationary distribution. Thus, we focus on the convergence to non-minimal quasi-stationary distributions.
	In contrast to the minimal case, there are few studies on the convergence \eqref{qsdConvergence} to non-minimal quasi-stationary distributions.
	We review the previous studies in more detail in Section \ref{section:prevStudies}.

	\subsection{Main results} \label{section:main}
	
	To state our main results, we prepare some notation.
	For a subset $S$ of $\bR$, we denote the set of probability measures on $S$ by $\cP (S)$ or $\cP S$.
	Let us fix an irreducible $\frac{d}{dm}\frac{d}{ds}$-diffusion $X$ on $(0,\infty)$ stopped at $0$ with $\bP_{x}[T_{0} < \infty] = 1$ for every $x > 0$.
	Define 
	\begin{align}
		\cP_{\Phi} := \{ \mu \in \cP(0,\infty) \mid \bE_\mu T_0^n < \infty \ ( n \geq 1) \}, \label{}
	\end{align}
	where $\bP_{\mu} := \int \bP_{x} \mu(dx)$.
	We will see in Proposition \ref{rd-Reccurence rel. of FHT} that the renewal transform $\Phi$ is well-defined as the map on $\cP_{\Phi}$; $\Phi$: $\cP_{\Phi} \to \cP_{\Phi}$.
	We denote the normalized $\alpha$-th moment of $T_{0}$ by $m_{\alpha}^{\mu}$ and a density function of $\Phi^{n}\mu$ w.r.t.\ $dm$ by $f_{n}^{\mu}$:
	\begin{align}
		m_{\alpha}^{\mu} := \frac{1}{\alpha!}\bE_{\mu}T_{0}^{\alpha} \quad (\alpha \in [1,\infty)), \quad f_{n}^{\mu}(x)dm(x) = \Phi^{n}\mu(dx), \label{}
	\end{align}
	where $\alpha! = \Gamma(\alpha + 1)$ for the Gamma function $\Gamma$.
	The existence of the density $f_{\mu}^{n}$ will be proven in Proposition \ref{rd-Rep. of density wrt dm}.	
	Let $\lambda_{0} \geq 0$ be the bottom of the $L^2$-spectrum of the generator $-\frac{d}{dm}\frac{d}{ds}$ with Dirichlet boundary condition at $0$ and Neumann boundary condition at $\infty$ if the boundary $\infty$ is regular (see Section \ref{section:FellerBdryClass} for the boundary classification).



	One of our main results is to give a sufficient condition for the convergence \eqref{qsdIRT}.
	Note that a necessary and sufficient condition for existence of infinitely many quasi-stationary distributions and the positivity of $\lambda_{0}$ will be given in Theorem \ref{QSD-char01}.

	\begin{Thm} \label{rd-Convergense of RDS}
		Let $\mu \in \cP_{\Phi}$. Assume that there exist infinitely many quasi-stationary distributions, and the following holds:
		\begin{align}
			\lim_{n \to \infty}\frac{m^{\mu}_{n-1}}{m^{\mu}_{n}} = \lambda \in (0,\lambda_{0}]. \label{rd-eq2} 
		\end{align}
		Then we have
		\begin{align}
			\lim_{n \to \infty}f_n^{\mu}(x) &= \lambda\psi_{-\lambda}(x) \quad (x > 0), \label{rd-eq14} \\
			\Phi^n\mu &\xrightarrow[n \to \infty]{} \nu_{\lambda},  \label{rd-eq15}
		\end{align}
		where a function $u = \psi_{\lambda} \ (\lambda \in \bC)$ is the unique solution of the following equation:
		\begin{align}
		\frac{d}{dm}\frac{d^+}{ds}u(x) = \lambda u(x), \quad 
		\lim_{x \to +0}u(x) = 0,\quad 
		\lim_{x \to +0}\frac{d^{+}}{ds}u(x) = 1 \quad (x \in (0,\infty)), \label{eq63}	
		\end{align}   
		and the probability measure $\nu_{\lambda} \ (\lambda \in (0,\lambda_{0}])$ is a quasi-stationary distribution defined by
		\begin{align}
			\nu_{\lambda}(dx) := \lambda \psi_{-\lambda}(x)dm(x). \label{} 
		\end{align}
	
	\end{Thm}

	\begin{Rem}
		The function $\psi_{\lambda}$ always exists since the process $X$ is irreducible and the boundary $0$ is regular or exit.
	\end{Rem}

	The other main result is to give a hierarchical structure of necessary conditions for the convergence \eqref{qsdConvergence}.
	
	\begin{Thm} \label{thm:hierarchyOfConds}
		Let $\mu \in \cP_{\Phi}$. Assume that there exist infinitely many quasi-stationary distributions.
		Let us consider the following conditions for $\lambda \in (0,\lambda_{0}]$:
		\begin{enumerate}
			\item $\mu_{t} \xrightarrow[t \to \infty]{} \nu_{\lambda}$.
			\item $\lim_{t \to \infty} \bP_{\mu}[T_{0} > t+s] / \bP_{\mu}[T_0 > t] = \mathrm{e}^{-\lambda s} \ (s > 0)$.
			\item $\lim_{n \to \infty} m^{\mu}_{n-1} / m^{\mu}_{n} = \lambda$.
			\item $\Phi^{n}\mu \xrightarrow[n \to \infty]{} \nu_{\lambda}$.
		\end{enumerate}
		Then the following implications hold: $(i) \Rightarrow (ii)$, $(ii) \Rightarrow (iii)$, $(iii) \Leftrightarrow (iv)$.
	\end{Thm}

	\subsection{Previous studies} \label{section:prevStudies}

	The renewal dynamical approach was intensively investigated in Ferrari, Kesten, Mart\'{\i}nez and Picco \cite{RenewalDynamicalApproach} to show the existence of the minimal quasi-stationary distribution for Markov chains on $\bN$.
	Under mild assumptions, their results ensure for wide range of Markov chains existence of the minimal quasi-stationary distribution. The present paper is strongly motivated by their study though we need a different approach since their arguments focus on the minimal quasi-stationary distribution.
	We also mention Takeda \cite{Takeda:QSD} as a general study on existence of the minimal
	quasi-stationary distribution for symmetric Markov process with the \textit{tightness property}.

	Yaglom limits of one-dimensional diffusions have long been studied.
	Mandl \cite{Mandl} gave a first remarkable result on this subject under the assumption of existence of a natural boundary. He gave a sufficient condition for the convergence to the minimal quasi-stationary distribution.
	His result has been extended and strengthened by many authors e.g., Collet, Mart\'{\i}nez and San Mart\'{\i}n \cite{CMM}, Hening and Kolb \cite{Hening}, Kolb and Steinsaltz \cite{Kolb}, Mart\'{\i}nez and San Mart\'{\i}n \cite{Ratesofdecay}, Littin \cite{Littin} and Cattiaux, Collet, Lambert, Mart\'{\i}nez, M\'{e}l\'{e}ard and San Mart\'{\i}n \cite{QSDpopulation}. These results show that under mild assumptions, convergence to the minimal quasi-stationary distribution follows for all compactly supported initial distributions.

	In contrast, there are few studies on Yaglom limits to non-minimal quasi-stationary distributions.
	In Mart\'{\i}nez, Picco and San Mart\'{\i}n \cite{DoAofBM}, they have considered Brownian motion with negative constant drifts: $X_{t} = B_{t} - ct \ (c > 0)$, where $B$ is a standard Brownian motion, and they gave a set of initial distributions whose Yaglom limit is a non-minimal quasi-stationary distribution.
	In Lladser and San Mart\'{\i}n, they studied Ornstein-Uhlenbeck processes: $dX_{t} = dB_{t} - cX_{t}dt \ (c > 0)$, and obtained the similar results.
	In \cite{QSD_unifying_approach}, a general approach for Yaglom limits to non-minimal quasi-stationary distributions was studied.
	One of the main results in \cite{QSD_unifying_approach} was to reduce the convergence \eqref{qsdConvergence} to the tail behavior of $T_{0}$ through the \textit{first hitting uniqueness}.
	For a set $\cP \subset \cP(0,\infty)$ of probability measures, 
	we say the first hitting uniqueness holds on $\cP$, when the following map is injective.
	\begin{align}
		\cP \ni \mu \longmapsto \bP_{\mu}[T_{0} \in dt]. \label{}
	\end{align}
	The following theorem gives the reduction under the first hitting uniqueness on
	\begin{align}
		\cP_{\mathrm{exp}} := \{ \mu \in \cP(0,\infty) \mid \bP_{\mu}[T_0 \in dt] = \lambda \mathrm{e}^{-\lambda t}dt \quad \text{for some } \lambda > 0 \}. \label{}
	\end{align}
	
	\begin{Thm}[{\cite[Theorem 1.1]{QSD_unifying_approach}}]
		Assume the first hitting uniqueness holds on $\cP_{\mathrm{exp}}$ and there exist $\lambda > 0$ and $\nu \in \cP(0,\infty)$ such that $\bP_{\nu}[T_0 \in dt] = \lambda \mathrm{e}^{-\lambda t}dt$.
		Then for $\mu \in \cP(0,\infty)$, the following are equivalent:
		\begin{enumerate}
			\item $
			\lim_{t \to \infty} \bP_{\mu}[T_0 > t + s] / \bP_{\mu}[T_0 > t] = \mathrm{e}^{-\lambda s} \ (s > 0)$.
			\item $\bP_{\mu_{t}}[T_0 \in ds] \xrightarrow[t \to \infty]{} \lambda \mathrm{e}^{-\lambda s}ds$.
			\item $\mu_t \xrightarrow
			[t \to \infty]{} \nu$.
		\end{enumerate}
	\end{Thm}
	Applying this theorem, in \cite[Theorem 1.2]{QSD_unifying_approach}, for Kummer diffusions with negative drifts, which are diffusions including the processes treated in \cite{DoAofBM} and \cite{DoAofOU}, a set of initial distributions whose Yaglom limit is non-minimal quasi-stationary distributions were given.
		
	\subsection*{Outline of the paper}
	
	The remainder of the present paper is organized as follows. 
	In Section \ref{section:preliminary}, we will recall several known results on one-dimensional diffusions and quasi-stationary distributions.
	In Section \ref{section:proof}, we will prove Theorem \ref{rd-Convergense of RDS} and Theorem \ref{thm:hierarchyOfConds}.

	\subsection*{Acknowledgement}

	The author would like to thank Kouji Yano, Servet Mart\'{\i}nez, Jaime San Mart\'{\i}n and Jorge Littin Curinao who read an early draft of the present paper and gave him valuable comments.
	This work was supported by JSPS KAKENHI Grant Number JP21J11000, JSPS Open Partnership Joint Research
	Projects grant no.\ JPJSBP120209921 and the Research Institute for Mathematical Sciences,
	an International Joint Usage/Research Center located in Kyoto University and, was carried out under the ISM Cooperative Research Program (2020-ISMCRP-5013).

	\section{Preliminaries}\label{section:preliminary}
	
	In this section, we recall several known results on one-dimensional diffusions and their quasi-stationary distributions. 
	
	\subsection{Feller's canonical form of second-order differential operators} \label{section:FellerBdryClass}
	
	Let $(X,\bP_x)_{x \in I}$ be a one-dimensional diffusion on $I = [0,\ell) \ \text{or} \ [0,\ell] \quad (0 < \ell \leq \infty)$, that is, the process $X$ is a time-homogeneous strong Markov process on $I$ which has a continuous path up to its lifetime.
	Throughout the present paper, we always assume an irreducibility in the following sense:
	\begin{align}
	\bP_x[T_y < \infty] > 0 \quad (x \in I \setminus \{0\}, \ y \in [0,\ell)), \label{eq37}
	\end{align} 
	where $T_y$ denotes the first hitting time of $y$. In addition, we assume $X$ certainly hits $0$ and the point $0$ is a trap;
	\begin{align}
		\bP_{x}[T_0 < \infty] = 1 \quad (x \in I \setminus \{0\}), \quad 
		X_t = 0 \quad \text{for} \ t \geq T_{0}. \label{trap}
	\end{align}
	
	Let us recall Feller's canonical form of the generator (see e.g., It\^o \cite[p.139]{Ito_essentials}).
	There exist a Radon measure $m$ on $I \setminus \{0\}$ with full support and a strictly increasing continuous function $s$ on $(0,\ell)$, and the local generator $\cL$ on $(0, \ell)$ is represented by
	\begin{align}
	\cL = \frac{d}{dm}\frac{d}{ds}. \label{}
	\end{align}
	The measure $m$ is called \textit{the speed measure} and $s$ is called \textit{the scale function} of $X$.
	We say $X$ is a \textit{$\frac{d}{dm}\frac{d}{ds}$-diffusion}.
	For a given second-order ordinary differential operator
	\begin{align}
		\cG = a(x)\frac{d^2}{dx^2} + b(x) \frac{d}{dx} \quad (x \in (0,\ell))
	\end{align}
	with $a(x) > 0$ on $(0,\ell)$, we can give, for example, its speed measure and its scale function by
	\begin{align}
		dm(x) := \frac{1}{a(x)}\exp \left(\int_{c}^{x}\frac{b(y)}{a(y)}dy \right)dx, \quad 
		ds(x) := \exp \left(-\int_{c}^{x}\frac{b(y)}{a(y)}dy\right)dx \quad 
 	\end{align}
	for an arbitrary taken $c \in (0,\ell)$.

	By using the speed measure $m$ and the scale function $s$, the boundaries of $I$ are classified.
	Let $\Delta = 0$ or $\ell$ and take $c \in (0,\ell)$ and set
	\begin{align}
		I(\Delta) = \int_{\Delta}^{c}ds(x)\int_{\Delta}^{x}dm(y), \quad
		J(c) = \int_{\Delta}^{c}dm(x)\int_{\Delta}^{x}ds(y). \label{}
	\end{align}
	The boundary $\Delta$ is classified as follows:
	 \begin{align}
	 	\Delta \ \text{is} \quad
	 	\left\{
	 	\begin{aligned}
	 		&\text{regular} &&\text{when}\  I(\Delta) < \infty, \ J(\Delta) < \infty. \\
	 		&\text{exit} &&\text{when} \ I(\Delta) = \infty, \ J(\Delta) < \infty. \\
	 		&\text{entrance} &&\text{when} \ I(\Delta) < \infty, \ J(\Delta) = \infty. \\
	 		&\text{natural} &&\text{when} \ I(\Delta) = \infty, \ J(\Delta) = \infty.
	 	\end{aligned}
	 	\right.
	 \end{align}
	From \eqref{eq37}, the boundary $0$ is always regular or exit in our setting.
	
	\subsection{Quasi-stationary distributions} \label{subsection:qsd}
	
	We recall a necessary and sufficient condition for existence of infinitely many quasi-stationary distributions.

	Define a function $u = \psi_{\lambda} \ (\lambda \in \bC)$ as the unique solution of the following equation:
	\begin{align}
	\frac{d}{dm}\frac{d^+}{ds}u(x) = \lambda u(x), \quad 
	\lim_{x \to +0}u(x) = 0,\quad 
	\lim_{x \to +0}\frac{d^{+}}{ds}u(x) = 1 \quad (x \in (0,\ell)), \label{}	
	\end{align}   
	where $\frac{d^+}{ds}$ denotes the right-differential operator by the scale function $s$. 
	Note that from the assumption that the boundary $0$ is regular or exit, the function $\psi_{\lambda}$ always exists.
	The operator $L = -\frac{d}{dm}\frac{d}{ds}$ defines a non-negative definite self-adjoint operator on $L^2(I,dm) := \{ f:I \to \bR \mid \int_{I}|f|^2dm < \infty \}$.
	Here we assume the Dirichlet boundary condition at $0$ and the Neumann boundary condition at $\ell$ if the boundary $\ell$ is regular. We denote the infimum of the spectrum of $L$ by $\lambda_0 \geq 0$.
	
	When the boundary $\ell$ is not natural, there exists a unique quasi-stationary distribution.
	Note that from the assumption \eqref{trap}, the boundary $\ell$ cannot be exit.
	\begin{Thm}[{see e.g., \cite[Lemma 2.2, Theorem 4.1]{Littin}}] \label{QSD-char2}
		Assume the boundary $\ell$ is regular or entrance.
		Then $\lambda_0 > 0$ and the function $\psi_{-\lambda_0}$ is strictly positive and integrable w.r.t.\ $dm$ and, there is a unique quasi-stationary distribution
		\begin{align}
			\nu_{\lambda_0}(dx) := \lambda \psi_{-\lambda_0}(x)dm(x) \label{}
		\end{align}
		with $\bP_{\nu_{\lambda_0}}[T_0 \in dt] = \lambda_0 \mathrm{e}^{-\lambda_0 t}dt$.
	\end{Thm}

	Recall the following property of the function $\psi_{-\lambda} \ ( \lambda > 0)$.

	\begin{Prop}[{\cite[Lemma 6.7, Lemma 6.18]{Quasi-stationary_distributions}}] \label{propofqsd}
		Suppose $\lambda_{0} > 0$, the boundary $\ell$ is natural and $s(\ell) = \infty$.
		Then for $\lambda > 0$ the following are equivalent:
		\begin{enumerate}
			\item $\lambda \in (0,\lambda_{0}]$.
			\item The function $\psi_{-\lambda}$ is non-negative on $[0,\ell)$.
			\item The function $\psi_{-\lambda}$ is strictly increasing on $[0,\ell)$.
			\item The function $\psi_{-\lambda}$ is strictly positive and integrable on $(0,\ell)$.
		\end{enumerate}
		Moreover, if one of these conditions holds, it holds
		\begin{align}
		1 = \lambda \int_{0}^{\ell}\psi_{-\lambda}(x)dm(x). \label{eq78}
		\end{align}
	\end{Prop}

	The following is a necessary and sufficient condition for
	existence of infinitely many quasi-stationary distributions.

	\begin{Thm} \label{QSD-char01}
		Suppose the boundary $\ell$ is natural.
		Then infinitely many quasi-stationary distributions exist if and only if
		\begin{align}
			\lambda_{0} > 0 \quad \text{and}\quad s(\ell) = \infty. \label{eq39}
		\end{align}
		The condition \eqref{eq39} is equivalent to 
		\begin{align}
		m(c,\ell) < \infty \quad \text{for some} \ c \in (0,\ell) \quad \text{and} \quad \limsup_{x \to \ell}s(x)m(x,\ell) < \infty. \label{eq42}			
		\end{align} 
		In this case, the set of quasi-stationary distributions is $\{ \nu_{\lambda} \}_{\lambda \in (0,\lambda_{0}]}$ for
		\begin{align}
			\nu_{\lambda}(dx) := \lambda \psi_{-\lambda}(x)dm(x), \label{eq73}
		\end{align}
		and it holds $\bP_{\nu_{\lambda}}[T_0 \in dt] = \lambda \mathrm{e}^{-\lambda t}dt$.
	\end{Thm}

	Though Theorem \ref{QSD-char01} can be shown by a combination of known results,
	we prove for completeness.

	\begin{proof}
		We may assume without loss of generality that $s(0) = 0$.
		The equivalence between \eqref{eq39} and \eqref{eq42} follows from \cite[Theorem 3 (ii), Appendix I]{KotaniWatanabe}.
		The fact that $\nu_\lambda \ ( \lambda \in (0,\lambda_0])$ is a quasi-stationary distribution can be shown by the same argument in \cite[Lemma 6.18]{Quasi-stationary_distributions}. Thus, we only show that every quasi-stationary distribution is given by $\nu_{\lambda}$ for some $\lambda \in (0,\lambda_0]$.
		
		Let $\mu$ be a quasi-stationary distribution with
		\begin{align}
			\bP_\mu[T_0 > t] = \mathrm{e}^{-\lambda t} \quad (\lambda > 0). \label{}
		\end{align}
		Recalling that one-dimensional diffusions have a transition density w.r.t.\ its speed measure,
		that is, there exists a jointly-continuous function $p(t,x,y)$ on $(0,\infty) \times (0,\ell)^2$ such that
		\begin{align}
			\bP_{x}[X_{t} \in dx, T_{0} > t] = p(t,x,y)dm(y) \label{}
		\end{align}
		(see e.g., McKean \cite{McKean:elementary}).
		Since it holds
		\begin{align}
			\mathrm{e}^{-\lambda t}\mu(A) = \bP_{\mu}[X_t \in A, T_{0} > t] = \int_{0}^{\ell}1_{A}(y)dm(y)\int_{0}^{\ell}p(t,x,y)\mu(dx),\label{}
		\end{align} 
		the probability measure $\mu$ is absolutely continuous w.r.t.\ $dm$, we denote the density by $\rho$. 
		Applying the well-known formula for non-negative measurable function $f$:
		\begin{align}
			\bE_x\left[\int_{0}^{T_0}f(X_t)1\{T_{0} > t\}dt \right] = \int_{0}^{\ell}(s(x)\wedge s(y))f(y)dm(y) \label{wellknownformula}
		\end{align}
		(see e.g., \cite[Theorem 49.1]{RogersWilliams2} and \cite[Lemma 23.10]{Kallenberg}),
		we have for a measurable set $A \subset (0,\ell)$
		\begin{align}
			\Phi \mu(A) &= \lambda\int_{0}^{\ell}\rho(x)dm(x)\int_{0}^{\ell} (s(x) \wedge s(y))1_{A}(y)dm(y) \label{} \\
			&= \lambda \int_{0}^{\ell} \left( \int_{0}^{x}ds(y)\int_{y}^{\ell}\rho(z)dm(z) \right)1_{A}(x)dm(x). \label{}
		\end{align}
		Since it holds $\Phi \mu = \mu$ from \eqref{eq6}, we obtain
		\begin{align}
			\rho(x) = \lambda \int_{0}^{x}ds(y)\int_{y}^{\ell}\rho(z)dm(z) \quad dm\text{-a.e.} \label{eq85}
		\end{align}
		Since $\int_{0}^{\ell}\rho(x)dm(x) = 1$, the function $u = \rho$ is the solution of the following equation
		\begin{align}
			u(x) = \lambda s(x) - \lambda \int_{0}^{x}ds(y)\int_{0}^{y}u(z)dm(z), \label{}
		\end{align}
		and hence $\rho(x) = \lambda\psi_{-\lambda}(x) \ dm$-a.e.
		From Proposition \ref{propofqsd}, it holds $\lambda \in (0,\lambda_{0}]$ and therefore $\mu = \nu_{\lambda}$. 
	\end{proof}

	From the proof of Theorem \ref{QSD-char01}, we can see the set of quasi-stationary distributions coincides with the set of fixed points of $\Phi$.
	\begin{Cor} \label{cor:QSDisFixedPoint}
		Assume there exists infinitely many quasi-stationary distributions.
		For $\mu \in \cP_{\Phi}$, $\Phi \mu = \mu$ if and only if $\mu = \nu_\lambda$ for some $\lambda \in (0,\lambda_0]$. 
	\end{Cor}

	The set $\{ \nu_{\lambda} \}_{\lambda \in (0,\lambda_{0}]}$ is totally ordered by a usual stochastic order, and $\nu_{\lambda_{0}}$ is the minimal element. 
	This is why we call $\nu_{\lambda_0}$ the minimal quasi-stationary distribution.

	\begin{Prop}[{\cite[Proposition 2.3]{QSD_unifying_approach}}] \label{prop:qsdOrdering}
		Suppose there exist infinitely many quasi-stationary distributions $\{ \nu_{\lambda} \}_{\lambda \in (0,\lambda_{0}]}$.
		Then it holds
		\begin{align}
			\nu_{\lambda}(x,\ell) \leq \nu_{\lambda'}(x,\ell) \quad (x \in (0,\ell),\ 0 < \lambda' \leq \lambda \leq \lambda_0). \label{}
		\end{align}
	\end{Prop}
	
	\section{Proof of the main results} \label{section:proof}

	For every $\frac{d}{d\tilde{m}}\frac{d}{ds}$-diffusion $X$ on $(0,\ell) \ (0 < \ell \leq \infty)$ satisfying \eqref{eq39}, the diffusion $s(X)$ is $\frac{d}{dm}\frac{d}{dx}$-diffusion on $(0,\infty)$ for $dm(x) = d\tilde{m}(s^{-1}(x))$. Thus, we may assume without loss of generality that the diffusion is under the natural scale: $s(x) = x$. 
	In this section, we only consider such natural scale diffusions on $(0,\infty)$.
	
	At first, we check that $\Phi$ preserves $\cP_{\Phi}$.
	\begin{Prop} \label{rd-Reccurence rel. of FHT}
		Let $\mu \in \cP_{\Phi}$.
		For $\alpha \in [1,\infty)$, we have
		\begin{align}
			\bE_{\Phi\mu}T_{0}^{\alpha} = \frac{\bE_\mu T_0^{\alpha + 1}}{(\alpha + 1) \bE_\mu T_0}. \label{rd-eq10}
		\end{align}
		More generally, we have for $0 \leq k \leq m$
		\begin{align}
			\bE_{\Phi^m\mu}T_{0}^{\alpha} = {\binom{\alpha+k}{\alpha}}^{-1}\cdot 
			\frac{\bE_{\Phi^{m-k}\mu}T_0^{\alpha + k}}{\bE_{\Phi^{m-k}\mu}T_0^k}. \label{rd-eq9}
		\end{align}
	\end{Prop}
	
	\begin{proof}
		Since \eqref{rd-eq9} follows from \eqref{rd-eq10} by induction, we only prove \eqref{rd-eq10}.
		From the Markov property, it holds
		\begin{align}
			\bE_{\Phi\mu}T_{0}^{\alpha} &= \alpha\int_{0}^{\infty}t^{\alpha-1}\bP_{\Phi\mu}[T_0 > t]dt \label{} \\
			&= \frac{\alpha}{\bE_{\mu}T_{0}}\int_{0}^{\infty}t^{\alpha-1}dt
			\int_{0}^{\infty}ds \int_{0}^{\infty} \bP_{x}[T_0 > t] \bP_{\mu}[X_{s} \in dx] \label{}  \\
			&= \frac{\alpha}{\bE_{\mu}T_{0}}\int_{0}^{\infty}t^{\alpha-1}dt
			\int_{0}^{\infty} \bP_{\mu}[T_{0} > t+s ] ds \label{}  \\
			&= \frac{1}{\bE_{\mu}T_{0}} \int_{0}^{\infty} s^{\alpha}\bP_{\mu}[T_{0} > s ] ds \label{} \\
			&= \frac{\bE_\mu T_0^{\alpha+1}}{(\alpha+1)\bE_\mu T_0}. \label{}
		\end{align}
	\end{proof}

	For the proof of Theorem \ref{rd-Convergense of RDS}, we need some preparation.
	For a function $g: (0,\infty) \to \bR$ with
	\begin{align}
		\int_{0}^{R}dy\int_{y}^{\infty}|g(z)|dm(z) < \infty \quad \text{for every } R > 0, \label{rd-eq12}
	\end{align}
	define an integral operator $K$ by
	\begin{align}
		Kg(x) := \int_{0}^{x}dy\int_{y}^{\infty}g(z)dm(z) = \int_{0}^{\infty}(x\wedge y)g(y)dm(y) \quad (x > 0). \label{}
	\end{align}
	
	Let us recall the formula \eqref{wellknownformula}. Then for a function $g$ with \eqref{rd-eq12}, it holds
		\begin{align}
			\bE_\mu \int_{0}^{T_0}g(X_t)dt = \int_{0}^{\infty}\mu(dx)\int_{0}^{\infty}(x\wedge y )g(y)dm(y) = \int_{0}^{\infty}Kg(x)\mu(dx). \label{rd-green function formula}
		\end{align}
	Applying \eqref{rd-green function formula}, we obtain the density function of $\Phi^{n}\mu$.	

	\begin{Prop}\label{rd-Rep. of density wrt dm}
		For $\mu \in \cP_{\Phi}$ and $n \geq 1$, there exists a density $f_n^{\mu}$ of $\Phi^n\mu$ w.r.t.\ $dm$; $\Phi^{n}\mu(dx) = f^{\mu}_{n}(x)dm(x)$. It is given by
		\begin{align}
			f_n^{\mu}(x) = \frac{1}{m_{n}^{\mu}}K^{n-1}G_\mu(x) \quad dm\text{-a.e.}, \label{}
		\end{align}
		where $G_\mu(x) := \int_{0}^{x}\mu(y,\infty)dy$ and we denote $K^\ell g:= K(K^{\ell -1} g) \  ( \ell \geq 1)$.
	\end{Prop}
	
	\begin{proof}
		From \eqref{rd-green function formula}, we have for a bounded measurable function $g$ with compact support on $(0,\infty)$,
		\begin{align}
			\int_{0}^{\infty}g(y)\Phi\mu(dx) 
			&= \frac{1}{\bE_\mu T_0}\int_{0}^{\infty}\mu(dx)\int_{0}^{\infty}(x\wedge y)g(y)dm(y) \label{} \\
			&= \frac{1}{\bE_\mu T_0}\int_{0}^{\infty}G_\mu(y)g(y)dm(y). \label{}
		\end{align}
		Since it holds that
		\begin{align}
			G_{\Phi\mu}(x) &= \int_{0}^{x}\Phi\mu(y,\infty)dy \label{} \\
			&= \frac{1}{\bE_{\mu}T_0}\int_{0}^{x}dy\int_{y}^{\infty}G_\mu(z)dm(z) \label{} \\
			&= \frac{1}{\bE_{\mu}T_0}KG_{\mu}(x) \label{rd-eq13}
		\end{align}
		it follows from Proposition \ref{rd-Reccurence rel. of FHT} and the (formal) self-adjointness of $K$ under $dm$ that
		\begin{align}
			&\int_{0}^{\infty}g(y)\Phi^n\mu(dy) \label{} \\
			=&  \frac{1}{\bE_{\Phi^{n-1}\mu} T_0}\int_{0}^{\infty}G_{\Phi^{n-1}\mu}(y)g(y)dm(y) \label{} \\
			=&  \frac{1}{(\bE_{\Phi^{n-1}\mu} T_0)(\bE_{\Phi^{n-2}\mu}T_0)}\int_{0}^{\infty}KG_{\Phi^{n-2}\mu}(y)g(y)dm(y) \label{} \\ 
			=&  \frac{1}{(\bE_{\Phi^{n-1}\mu} T_0)(\bE_{\Phi^{n-2}\mu}T_0)}\int_{0}^{\infty}G_{\Phi^{n-2}\mu}(y)Kg(y)dm(y) \label{} \\
			=& \cdots \label{} \\
			=&  \frac{1}{(\bE_{\Phi^{n-1}\mu}T_0)(\bE_{\Phi^{n-2}\mu} T_0)\cdots(\bE_{\Phi\mu} T_0)(\bE_{\mu}T_0)}\int_{0}^{\infty}K^{n-1}G_{\mu}(y)g(y)dm(y) \label{} \\
			=& \frac{n!}{\bE_{\mu}T_0^n}\int_{0}^{\infty}K^{n-1}G_{\mu}(y)g(y)dm(y). \label{}
		\end{align}
		The proof is complete.
	\end{proof}
	
	Now we prove Theorem \ref{rd-Convergense of RDS}.
	\begin{proof}[Proof of Theorem \ref{rd-Convergense of RDS}]
		For a function $g$ with
		\begin{align}
			\int_{0}^{R}dy\int_{0}^{y}|g(z)|dm(z) < \infty \label{}
		\end{align}
		for every $R > 0$, define an integral operator $I$ by
		\begin{align}
			Ig(x) := \int_{0}^{x}dy \int_{0}^{y}g(z)dm(z) \quad (x > 0). \label{}
		\end{align}
		Let $g \in L^1((0,\infty),dm)$.
		We have
		\begin{align}
			Kg(x) &= \int_{0}^{x}dy\int_{y}^{\infty}g(z)dm(z) \label{} \\
			&= x\int_{0}^{\infty}g(z)dm(z) -Ig(x). \label{}  
		\end{align}
		Then it follows that
		\begin{align}
			K^{n-1}G_{\mu}(x) &= x\int_{0}^{\infty}K^{n-2}G_\mu(y)dm(y) - IK^{n-2}G_\mu(x) \label{} \\
			&= m^{\mu}_{n-1}x - IK^{n-2}G_\mu(x) \label{} \label{} \\
			&= m^{\mu}_{n-1}x - I( m^{\mu}_{n-2}x - IK^{n-3}G_\mu(x) ) \label{} \\
			&= m^{\mu}_{n-1}x - m^{\mu}_{n-2}Ix + I^2K^{n-3}G_\mu(x) \label{} \\
			&= \cdots \label{} \\
			&= \sum_{k=1}^{n-1}(-1)^{k-1}m^{\mu}_{n-k}I^{k-1}x + (-1)^{n-1}I^{n-1}G_\mu(x). \label{}
		\end{align}
		Then we have
		\begin{align}
			f_n^{\mu}(x) = \sum_{k=1}^{n-1}(-1)^{k-1}\frac{m^{\mu}_{n-k}}{m^{\mu}_{n}}I^{k-1}x + (-1)^{n-1}\frac{1}{m^{\mu}_{n}}I^{n-1}G_\mu(x). \label{rd-eq3}
		\end{align}
		From \eqref{rd-eq2} we have $M := \sup_{n \geq 0}\frac{m^{\mu}_{n-1}}{m^{\mu}_n} < \infty$,
		where we denote $m^{\mu}_0 = 1$.
		Then it follows that
		\begin{align}
			\sum_{k=1}^{n-1} \frac{m^{\mu}_{n-k}}{m^{\mu}_{n}}I^{k-1}x \leq \sum_{n=1}^{\infty}M^{n}I^{n-1}x = M\psi_{M}(x) < \infty. \label{rd-eq4}
		\end{align}
		Next we show the second term in the RHS of \eqref{rd-eq3} vanishes as $n \to \infty$.
		It is not difficult to check that
		\begin{align}
			I^{n}G_\mu(x) \leq  \frac{1}{n!} x\left( \int_{0}^{x}ydm(y) \right)^n, \label{rd-eq5}
		\end{align}
		and since $m^{\mu}_{n} = \prod_{1 \leq i \leq n} (m_{i}^{\mu} / m_{i-1}^{\mu}) \geq M^{-n}$,
		we obtain for every $R > 0$
		\begin{align}
			\lim_{n \to \infty}\sup_{x \in [0,R]}\frac{1}{m^{\mu}_{n}}I^{n-1}|G_\mu(x)| = 0. \label{}
		\end{align}
		Then from \eqref{rd-eq3} and the dominated convergence theorem, we have
		\begin{align}
			\lim_{n \to \infty}f_n^{\mu}(x) = \lambda\psi_{-\lambda}(x). \label{}
		\end{align}
		From \eqref{rd-eq4} and \eqref{rd-eq5}, we have 
		\begin{align}
			f_n^{\mu}(x) \leq M\psi_{M}(x) + \frac{M^{n}}{(n-1)!}x \left( \int_{0}^{x}ydm(y) \right)^{n-1} \leq M\psi_{M}(x) + Mx \mathrm{e}^{M \int_{0}^{x}ydm(y)} \label{}
		\end{align}
		and, it is obvious that 
		\begin{align}
			\int_{0}^{R}(M\psi_{M}(x) + Mx \mathrm{e}^{M \int_{0}^{x}ydm(y)})dm (x)< \infty \label{}
		\end{align}
		for every $R > 0$. Hence, from the dominated convergence theorem, it follows that
		\begin{align}
			\Phi^n \mu \xrightarrow[n \to \infty]{} \nu_{\lambda}. \label{}
		\end{align}
	\end{proof}	

	For the proof of Theorem \ref{thm:hierarchyOfConds}, we prepare a continuity result for the transform $\Phi$.

	\begin{Prop} \label{prop:contOfPhi}
		Let $\mu_{n}, \mu \in \cP(0,\infty)$ such that $\bE_{\mu_{n}}T_0, \ \bE_{\mu}T_{0} < \infty$.
		Suppose $\mu_{n} \xrightarrow[n \to \infty]{} \mu$.
		Then the following are equivalent:
		\begin{enumerate}
			\item $\bE_{\mu_{n}}T_{0} \xrightarrow[n \to \infty]{} \bE_{\mu} T_{0}$.
			\item $\Phi \mu_{n} \xrightarrow[n \to \infty]{} \Phi \mu$.
		\end{enumerate}
	\end{Prop}

	\begin{proof}
		Let $f$ be a continuous function with a compact support on $(0,\infty)$.
		From \eqref{wellknownformula} it holds
		\begin{align}
			(\bE_{\mu_{n}}T_{0})\Phi \mu_{n}(f) &= \int_{0}^{\infty}\bP_{\mu_{n}}[f(X_{t})1\{ T_{0} > t \}]dt \label{} \\
			&= \int_{0}^{\infty}\mu_{n}(dx) \int_{0}^{\infty}f(y) (x \wedge y)dm(y). \label{}
		\end{align}
		Since the function $f$ is compactly supported, the function $Kf(x) := \int_{0}^{\infty}f(y)(x \wedge y)dm(y) \ (x > 0)$ is bounded continuous.
		The rest of the proof is obvious.
	\end{proof}

	We prove Theorem \ref{thm:hierarchyOfConds}.
		
	\begin{proof}[Proof of Theorem 1.3]
		The implication (i) $\Rightarrow$ (ii) is obvious.
		We first show (ii) $\Rightarrow$ (iii).
		Note that when we set $g(t) := \bP_{\mu}[T_0 > \log t]$, the condition (ii) is equivalent to
		\begin{align}
			\lim_{t \to \infty}\frac{g(st)}{g(t)} = s^{-\lambda} \quad (s > 0), \label{} 
		\end{align}
		the regular variation of the function $g$ at $\infty$ of order $-\lambda$.
		Then from Karamata's theorem \cite[Proposition 1.5.10, Theorem 1.6.1]{Regularvariation},
		the condition (ii) is equivalent to the following:
		\begin{align}
			h(t) := \frac{1}{\bP_{\mu}[T_0 > t]}\int_{t}^{\infty}\bP_{\mu}[T_0 > s]ds
			\xrightarrow{t \to \infty} \frac{1}{\lambda}. \label{eq7}
		\end{align}
		From Fubini's theorem, it holds for $n \geq 2$
		\begin{align}
			m_n^{\mu} = \frac{1}{(n-2)!}\int_{0}^{\infty}t^{n-2}\bP_{\mu}[T_0 > t]h(t)dt. \label{}
		\end{align}
		For $R > 0$ it is not difficult to see
		\begin{align}
			\frac{\int_{0}^{R}t^{n-2}\bP_{\mu}[T_0 > t]h(t)dt}{\int_{R}^{\infty}t^{n-2}\bP_{\mu}[T_0 > t]h(t)dt} \xrightarrow{n\to \infty} 0.\label{}
		\end{align}
		Thus, from \eqref{eq7} it follows
		\begin{align}
			\lim_{n \to \infty}\frac{m_{n}^\mu}{m_{n-1}^\mu} = \lim_{n \to \infty}\frac{\int_{0}^{\infty}t^{n-2}\bP_{\mu}[T_0 > t]h(t)dt}{\int_{0}^{\infty}t^{n-2}\bP_{\mu}[T_0 > t]dt} = \frac{1}{\lambda}. \label{}
		\end{align}

		Since we have already shown (iii) $\Rightarrow$ (iv) in Theorem \ref{rd-Convergense of RDS},
		we finally show (iv) $\Rightarrow$ (iii).
		Set $\mu_{n} := \Phi^{n} \mu$.
		Since it holds $\mu_{n} \xrightarrow[n \to \infty]{} \nu_{\lambda}$ and $\Phi \mu_{n} = \mu_{n+1} \xrightarrow[n \to \infty]{} \Phi\nu_{\lambda} = \nu_{\lambda}$,
		we have from Proposition \ref{prop:contOfPhi} and Proposition \ref{rd-Reccurence rel. of FHT}
		\begin{align}
			\frac{m_{n+1}^{\mu}}{m_{n}^{\mu}} = \bE_{\mu_{n}}T_{0} \xrightarrow[n \to \infty]{} \frac{1}{\lambda}. \label{}
		\end{align}
	\end{proof}

	\begin{Rem}
		The condition (iii) in Theorem \ref{thm:hierarchyOfConds} implies
		\begin{align}
			\lim_{t \to \infty} \frac{1}{t} \log \bP_{\mu}[T_{0} > t] = - \lambda. \label{eq11}
		\end{align}
		Indeed, under the condition (iii), we see from Proposition \ref{rd-Reccurence rel. of FHT} that
		$\lim_{n \to \infty} m_{j}^{\Phi^{n} \mu} = j! / \lambda^{j}$ for $j \geq 1$,
		which implies $\bP_{\Phi^{n}\mu}[T_{0} \in dt] \xrightarrow[n \to \infty]{} \lambda \mathrm{e}^{-\lambda t}dt$ and $m^{\mu}_{\alpha} / m^{\mu}_{\alpha+1} \xrightarrow[\alpha \to \infty]{} \lambda$.
		Then it follows
		\begin{align}
			\frac{1}{\alpha} \log m^{\mu}_{\alpha} = \frac{1}{\alpha} \sum_{0 \leq j < \floor{\alpha}} \log (m^{\mu}_{\alpha - j } / m^{\mu}_{\alpha - j- 1}) \xrightarrow[\alpha \to \infty]{} - \log \lambda. \label{eq16}
		\end{align}
		Then applying a Tauberian theorem of exponential type \cite[Theorem 1]{Tauberian_exp_type} to the function
		\begin{align}
			F(\alpha) := \int_{0}^{\infty} \mathrm{e}^{\alpha \log (t / (\alpha!)^{1/\alpha})}\bP_{\mu}[T_{0} \in dt] \  (= m^{\mu}_{\alpha}), \label{} 
		\end{align}
		we see \eqref{eq11} and \eqref{eq16} are equivalent.
	\end{Rem}

\bibliography{arxiv01.bbl} 
\bibliographystyle{plain}

\end{document}